\documentclass[a4paper]{article}

\usepackage[utf8]{inputenc}
\usepackage[T1]{fontenc}

\usepackage{amstext}
\usepackage{amsfonts}
\usepackage{amsmath}
\usepackage{amssymb}
\usepackage{amsthm}

\newtheorem{theorem}{Theorem}[section]
\newtheorem{proposition}[theorem]{Proposition}
\newtheorem{lemma}[theorem]{Lemma}
\newtheorem{definition}[theorem]{Definition}

\newtheorem{corollary}[theorem]{Corollary}

\newtheorem*{theoremA}{Theorem A}
\newtheorem*{theoremB}{Theorem B}

\newtheorem*{definitions}{Definition}
\newtheorem*{remark}{Remark}
\newcommand{\Dim}{\mathrm{dim}}
\newcommand{\DimT}{\underline{\mathrm{dim}}^{\mathrm{T}}}
\newcommand{\Mrad}{{\mathcal{M}}_{\mathrm{r}}^{\mathrm{1}}}
\newcommand{\R}{\mathbf{R}}
\newcommand{\diminf}{\underline{\dim}}

\newcommand{\BMS}{m_{\mathrm{BMS}}}

\newcommand{\Hr}{\mathbf{H}}

\author{Laurent Dufloux}
\title{Projections of Patterson-Sullivan Measures and the Dichotomy of Mohammadi-Oh}
\begin{document}
\maketitle

\abstract{Let $\Gamma$ be some discrete subgroup of $\mathbf{SO}^o(n+1,\R)$ with finite Bowen-Margulis-Sullivan measure. We study the dynamics of the Bowen-Margulis-Sullivan measure measure with respect to closed connected subspaces of the $N$ component in some Iwasawa decomposition $\mathbf{SO}^o(n+1,\R)=KAN$. We also study the dimension of projected Patterson-Sullivan measures along some fixed direction.}

\section{Introduction}
\subsection{Statement of results}
Fix some integer $n \geq 2$ and let $G$ be the group of direct
isometries of the real hyperbolic $(n+1)$-space $\Hr^{n+1}$,
$G=\mathbf{SO}^o(1,n+1)$, and choose some Iwasawa decomposition
$G=KAN$. Recall that $N$ identifies with the real $n$-space $\R^n$. We
will consider the Bowen-Margulis-Sullivan (BMS) measure on $\Gamma
\backslash G$. Our first result in this paper is the following
\begin{theoremA}
  Let $\Gamma$ be a discrete non-elementary subgroup of $G$ of growth
  exponent $\delta_\Gamma$, and assume that $\Gamma$ is Zariski-dense
  and has finite BMS measure. Let $m$ be some integer, $1 \leq m \leq
  n$, and fix some $m$-plane $U$ in $N$. Let $U$ act on the right on
  $\Gamma \backslash G$. The following dichotomy holds:
  \begin{itemize}
  \item if $\delta_\Gamma \leq n-m$, the BMS measure is totally
    dissipative (and thus not ergodic) with respect to $U$ ;
  \item if $\delta_\Gamma > n-m$, the BMS measure is totally recurrent
    with respect to $U$.
  \end{itemize}
\end{theoremA}
Unless $\Gamma$ is a lattice, the BMS measure is
not $N$-invariant (or even quasi-invariant), so the reader may wonder
what it means for a non-invariant measure to be totally recurrent or
dissipative -- see section \ref{ss.recurrence}.

Theorem A above is a particular case of the more precise Theorem
\ref{th.principal1}, which I do not state in this introduction because
it is slightly more technical, as it involves the theory of
conditional measures along group operations and the dimension theory
of such conditional measures. Theorem A is a \emph{qualitative}
statement which is weaker than the \emph{quantitative} Theorem
\ref{th.principal1}, the latter dealing with \emph{dimension} of BMS
measure along the subgroup $U$ (as well as the transversal dimension
with respect to this subgroup).

Before stating our second result, let us introduce the (\emph{ad hoc})
notion of \emph{regular measure} on the Euclidean space.
\begin{definitions}
  Let $\mu$ be some (Borel) probability measure on $\R^n$ ($n \geq
  2$). Assume that $\mu$ has exact dimension $\delta$. We say that
  $\mu$ is \emph{regular} if for any $m$-plane $V$ in $\R^n$ ($1 \leq
  m \leq n-1$) the orthogonal projection of $\mu$ onto $V$ has
  dimension $\inf\{\delta,m\}$ almost everywhere.
\end{definitions}
Obviously, if this is the case, then the orthogonal projection of
$\mu$ onto $V$ is in fact exact dimensional.

\begin{theoremB}[Theorem \ref{th.principal2}]
  Let $\Gamma$ be a discrete non-elementary subgroup of
  $G=\mathbf{SO}^o(1,n+1)$. Assume that $\Gamma$ is Zariski-dense and
  has finite BMS measure. Let $\mu$ be the
  Patterson-Sullivan measure (of exponent $\delta_\Gamma$) associated
  with $\Gamma$. For $\mu$-almost every $\xi \in \partial \Hr^{n+1}$,
  the push-forward of $\mu$ through the inverse stereographic mapping
  $\partial \Hr^{n+1} \setminus \{\xi\} \to \R^n$ is a regular
  measure.
\end{theoremB}
The proof of Theorem B relies on results of Hillel
Furstenberg, Pablo Shmerkin and Michael Hochman.

\subsection{Background and motivation}
Let us now provide some background. Theorem A is motivated by the works of
Mohammadi-Oh and Maucourant-Schapira. The seminal paper is \cite{OhMohammadi}. In this work, Mohammadi and Oh look at the dynamics of the Burger-Roblin measure (BR, see \cite{OhMohammadi} for the definition) on $\Gamma \backslash G$ with respect to a fixed $m$-plane $U$ in $N$. They state and prove the following
\begin{theorem}[\cite{OhMohammadi}, Theorem 1.1]
Let $n=2$ (so we work in $\Hr^3$). Assume that $\Gamma$ is Zariski-dense and convex-cocompact. If $\delta_\Gamma > 1$, then the Burger-Roblin measure is totally recurrent and ergodic with respect to any 1-parameter subgroup $U$ of $N$.
\end{theorem}
One of the features of their approach is the use of Marstrand's projection Theorem to show that the BMS measure is recurrent with respect to $U$ (when $\delta_\Gamma > 1$), which  implies that the BR measure is also recurrent with respect to $U$. The authors are then able to deduce that the BR measure is ergodic with respect to $U$; the proof is difficult and lengthy and we refer the reader to \cite{OhMohammadi}.

This is what prompted me to try and analyse precisely the BMS measure from the geometric and dynamical point of view, in order to deduce corresponding statements for the Burger-Roblin measure. In particular, the case $\delta \leq 1$ is not, in my opinion, clearly settled in \cite{OhMohammadi}.

Maucourant and Schapira have also looked at these questions. They prove the following result.
\begin{theorem}[\cite{MaucourantSchapira}]
  Let $n$ be $\geq 2$. Assume that $\Gamma$ is Zariski-dense and has finite BMS measure. Let $U$ be some $m$-plane of $N$, $1 \leq m \leq n$.
  \begin{enumerate}
  \item If $\delta_\Gamma < n-m$ and, furthermore, $\Gamma$ is convex-cocompact, then the BMS and BR measures are totally dissipative.
    \item If $\delta_\Gamma > n-m$, the BMS and BR measures are ergodic and recurrent.
  \end{enumerate}
\end{theorem}
Their approach is geometric and it also relies on Marstrand's projection Theorem. Note that the case $\delta_\Gamma = n-m$ is left open.

Theorem A thus clarifies the situation for the BMS measure. Note that we do not deal here with the BR measure; in fact, this will be done in a subsequent paper. It should be noted that, although recurrence of the BMS measure clearly implies recurrence of the BR measure, there is no reason why dissipativity of the BMS measure should imply dissipativity of the BR measure; indeed, it will be seen that the BR measure is in fact recurrent with respect to any $m$-plane when $\delta_\Gamma = n-m$, at least if $\Gamma$ is convex-cocompact. 

Another motivation for Theorem A (more precisely, for Theorem \ref{th.principal1}) comes from the complex hyperbolic world. We refer to \cite{preprint} for more details on this. To put it shortly, most of the Patterson-Sullivan theory we recall below holds, \emph{mutatis mutandis}, when the real hyperbolic space is replaced with the complex hyperbolic space; in this setting, $N$ identifies not with the Euclidean space but with the Heisenberg space, and the dimension of the BMS measure along the center of $N$ is related to the dimension of the limit set of $\Gamma$ \emph{with respect to the spherical metric on the boundary}. See \cite{preprint}, Theorem 37. The reason why the proof of Theorem \ref{th.principal1} does not easily translate into the complex hyperbolic setting is because of the lack of a useful version of Marstrand's projection Theorem in Heisenberg space.

As for Theorem B, the question it answers has apparently never been considered before. In general, understanding the geometry (and, in particular, the dimension) of the projection of a given ``fractal'' set or measure along some \emph{fixed} direction is a difficult problem . In recent years, a powerful theory has been set up by Furstenberg, Hochman, Shmerkin and other people. See \cite{Hoch} and other references there. We will apply this theory to prove Theorem B. 

In all this paper we fix an Iwasawa decomposition $G=KAN$. Recall that
$K$ is isomorphic to $\mathbf{SO}(n+1)$, $A$ is isomorphic to $\R$
(since $\Hr^{n+1}$ is a rank one symmetric space) and $N$ is
isomorphic to $\R^{n}$.

Recall that $A$ and $M$ centralize each other and that the group $AM$
normalizes the ``horospherical group'' $N$; we may thus look at the
way $AM$ operates on $N$, \emph{i.e.} for any $g \in AM$ consider the
group automorphism $h \mapsto ghg^{-1}$ from $N$ to $N$. In fact, $N$
may be identified with $\R^n$ and $M$ then identifies with the group
of all linear isometries of $\R^n$ (endowed with the usual euclidean
norm). Furthermore, $A$ identifies with the group of all linear
homotheties of $N$. These elementary facts are to be borne in mind as
we will use them throughout this paper.

If $G$ is some group and $g$ any element of $G$, the left and right
translations by $g$ and $g^{-1}$ respectively are denoted by
\[ L_g : h \mapsto gh\quad ; \quad R_g : h \mapsto h g^{-1} \text. \]

The plan of the paper is as follows. In section 2 we recall the
classical Marstrand Theorem as well as some useful facts from the
theory of self-similar measures. In section 3 we state basic facts
from the Patterson-Sullivan theory, and we apply the celebrated
Ledrappier-Young Theorem to the disintegration of the
BMS measure along subgroups of $N$. In section 4
and 5 we state and prove our main results.

I would like to thank my thesis advisor, Jean-Fran\c{c}ois Quint, for his constant support and help during my PhD.

\section{Preliminaries}
\subsection{Definition}
\begin{definition}
Let $X$ be a metric space and $\mu$ a  measure on $X$ such that any ball in $X$ has finite measure. The \emph{lower dimension} of $\mu$ at some point $x \in X$ is the finite or infinite number
\[ \diminf(\mu,x) = \liminf_{\rho \to 0} \frac{\log \mu(B(x,\rho))}{\log \rho}. \]
The \emph{lower dimension} of $\mu$ is the $\mu$-essential infimum of $\diminf (\mu,x)$, \emph{i.e.}
\[ \diminf(\mu) = \sup\{ s \geq 0\ ;\ \diminf(\mu,x) \geq s \text{ for $\mu$-almost every $x$} \}. \]
We say that $\mu$ is \emph{exact-dimensional of dimension $\delta$} if, for $\mu$-almost every $x$
\[ \lim_{\rho \to 0} \frac{\log \mu(B(x,\rho))}{\log \rho} = \delta \text. \]
\end{definition}

Let us state for future reference the following obvious
\begin{lemma}
Let $X,Y$ be two metric spaces and $\pi : X \to Y$ a Lipschitz mapping. If $\mu$ is a probability measure on $X$, then 
\[ \diminf(\mu,x) \geq \diminf(\pi \mu,\pi x) \]
for any $x \in X$, and
\[ \diminf(\mu) \geq \diminf(\pi \mu) \text. \]
\end{lemma}

In this lemma, as well as in the rest of this paper, we denote by $\pi \mu$ the push-forward of $\mu$ through $\pi$, \emph{i.e.} $\pi \mu(A) = \mu (\pi^{-1}A)$ for any Borel subset $A \subset Y$.

Let us also introduce the following definition for the sake of brevity.
\begin{definition}
If $G$ is a Borel group and $H$ a Borel subgroup of $G$, we say that a Borel measure $\mu$ on $G$ is concentrated on a Borel graph over $G/H$ if there is a Borel section $\Sigma$ (with respect to $H$) that has full $\mu$-measure.
\end{definition}
Recall that $\Sigma$ is a section with respect to $H$ if $\Sigma H=G$ and $\Sigma \cap gH = \{g\}$ for any $g \in \Sigma$.

\subsection{Almost sure dimension of projections}
Here we state the classical Theorem of Marstrand on the almost sure dimension of projected measures in Euclidean spaces. Fix some integer $n \geq 2$. For any integer $m$, $1 \leq m \leq n-1$, and any $m$-plane $V$ in $\R^n$, we denote by $\pi_V$ the orthogonal projection $\R^n \to V$. If $\mu$ is a finite measure on $\R^n$, its push-forward through $\pi_V$ is denoted by $\pi_V \mu$.

Recall that the Grassmannian of $m$-planes of $\R^n$ is a Hausdorff compact topological space. It carries a unique probability measure that is invariant under the natural operation of the orthogonal group $O(n)$. When we use the phrase ``for almost every $m$-plane $V$ in $\R^n$'' we mean \emph{with respect to the above probability measure}.

\begin{proposition}[\cite{Mattila}] \label{prop.marstrand}
Let $\mu$ be some probability measure on $\R^n$ and fix some integer $m$, $1 \leq m \leq n-1$. For almost every $m$-plane $V$ in $\R^n$, 
\[ \diminf(\pi_V \mu) = \inf\{m, \diminf(\mu) \} \]
and in the case when $\diminf(\mu) > m$, the projected measure $\pi_V \mu$ is almost surely absolutely continuous (with respect to the Haar measure on $V$).

Furthermore, if $\mu$ is exact-dimensional, then so is $\pi_V \mu$ (almost surely).
\end{proposition}
\subsection{Projections of self-similar measures}
We are going to state an important result of M. Hochman about Hausdorff dimension of projections of self-similar measures. First, we have to pass through several definitions and  notations. The reader is referred to \cite{Hoch} for more details.

\subsubsection{Notations}
We  work in Euclidean $n$-space $\R^n$. The unit ball with respect to the supremum norm, $[-1,1]^n$, is denoted by $B_1$. Let $\mathcal M$ be the space of non-zero Radon measures. For any real number $t$, let $S_t : \R^n \to \R^n$ be the homothety
\[ S_t : x \mapsto e^t x \]
which induces a mapping $S_t : \mathcal M \to \mathcal M$
\[ \mu \mapsto S_t \mu\text. \]
The 'S' stands for \emph{scaling}.

If $\mu \in \mathcal M$ is such that $B_1$ is non-negligible, we denote by $\mu^*$ the normalized measure
\[ \mu^* = \frac{\mu}{\mu(B_1)} \text. \]
and by $\mu^\square$ the conditional measure defined by
\[ \mu^\square(A) = \frac{\mu(A \cap B_1)}{\mu(B_1)} \text. \]

The set of all measures $\mu \in \mathcal M$ such that $\mu(B_1)=1$ is denoted by $\mathcal M^*$; the set of all probability measures $\mu \in \mathcal M$ supported on $B_1$ is denoted by $\mathcal M^\square$.

If $\mu \in \mathcal M$ is such that $0$ belongs to the support of $\mu$, we may consider $(S_t \mu)^*$ and $(S_t \mu)^\square$ for any real number $t$. We then write
\[ S_t^* \mu = (S_t \mu)^*,\quad S_t^\square \mu = (S_t \mu)^\square \text. \]

For any $x \in \R^n$, let $T_x$ be the translation $y \mapsto y-x$. If $\mu$ belong to $\mathcal M$, we write $T_x^* \mu = (T_x \mu)^*$; this is well-defined for any $x$ in the support of $\mu$.

\subsubsection{Fractal distributions}
Following Hochman \cite{Hoch}, probability measures on spaces of measures are called \emph{distributions}, and denoted by the letter $P$. We keep the letter $\mu$ for measures on Euclidean spaces. For example, if $\mu$ is some element of $\mathcal M$, the Dirac measure at $\mu$, $\delta_\mu$, is a distribution on $\mathcal M$.

\begin{definition}
A \emph{fractal distribution} is a distribution on $\mathcal M^*$ that satisfies the following conditions:
\begin{enumerate}
\item given any relatively compact neighbourhood $U$ of $0$ in $\R^n$, the distribution
\[ \int \mathrm{d} P(\mu) \int_U \delta_{T_x^* \mu} \mathrm{d} \mu(x) \]
on $\mathcal M^*$ is equivalent to $P$ (which means that each of them is absolutely continuous with respect to the other);
\item for any real number $t$, $S_t^* P=P$; in other words, $P$ is invariant under the flow $(S_t^*)_t$.
\end{enumerate}
If, furthermore, $P$ is ergodic with respect to $(S_t^*)$, we say that $P$ is an ergodic fractal distribution.
\end{definition}

Note that if $P$ satisfies condition 1 (such a distribution $P$ is called \emph{quasi-Palm}), its image $P^\square$ through the (not everywhere defined) mapping $\mu \mapsto \mu^\square$ is a well-defined distribution on $\mathcal M^\square$. Indeed, condition 1 implies that $0$ belongs to the support of $\mu$ for $P$-almost every $\mu$.

A distribution on $\mathcal M^\square$ that may be written $P^\square$, where $P$ is an ergodic fractal distribution, is called a \emph{restricted ergodic fractal distribution}. 

\subsubsection{Uniformly scaling measures}
Let $\mu$ be some Radon measure on $\R^n$ and let $x$ some point in the support of $\mu$. We denote by $\mu_{x,t}$ the image of $\mu$ through the composition $S_t \circ T_x$ ($t \in \R$) and  consider the probability measure $\mu_{x,t}^\square$. Let $\langle \mu \rangle_{x,T}$ be the distribution on $\mathcal M^\square$
\[ \langle \mu \rangle_{x,T} = \frac{1}{T} \int_0^T \delta_{\mu_{x,t}^\square} \mathrm{d}t \text. \]

If there exists a distribution $P$ on $\mathcal M^\square$ such that, for $\mu$-almost every $x$, $\langle \mu \rangle_{x,T}$ converges (weakly) to $P$ as $T \to \infty$, we say that $\mu$ is \emph{uniformly scaling}, and that it \emph{generates} $P$.

The following non-obvious Theorem is stated for clarity, as we will not need it.

\begin{theorem}[\cite{Hoch}]
If $\mu$ is a uniformly scaling measure generating a distribution $P$, $P$ is a restricted ergodic fractal distribution. Conversely, if $P$ is an ergodic fractal distribution, then $P$-almost every $\mu$ is a uniformly scaling measure generating $P^\square$.
\end{theorem}

\subsubsection{Mean dimension of a fixed projection}
Let $P$ be a distribution on $\mathcal M^\square$. Let $\pi$ be some linear mapping from $\R^n$ onto $\R^m$, $1 \leq m \leq n-1$. Let
\[ E_P(\pi) = \int \mathrm{d} P(\mu)\ \diminf(\pi \mu) \text. \]
\begin{theorem}[\cite{Hoch}] \label{th.hochman}
  Let $\mu$ be a uniformly scaling probability measure on $\R^n$ and let $P$ be the restricted ergodic fractal distribution generated by $\mu$. For any linear mapping $\pi$ from $\R^n$ onto $\R^m$, $1 \leq m \leq n-1$, 
\[ \diminf (\pi \mu) \geq E_P(\mu) \text. \]
\end{theorem}

\begin{corollary} \label{corollary.lol}
Assume, furthermore, that $P$ is invariant under the natural operation of the special orthogonal group $\mathbf{SO}(n)$. Then 
\[ \diminf (\pi \mu) = \inf \{ m, \diminf (\mu) \} \text. \]
\end{corollary}
\begin{proof}
Apply proposition \ref{prop.marstrand} and Fubini's Theorem.
\end{proof}

\section{Patterson-Sullivan theory}
We recall some classical results and definitions. Let $G=\mathbf{SO}^o(1,n+1)$ ($n \geq 2$) be the group of direct isometries of the real hyperbolic $(n+1)$-plane $\Hr^{n+1}$. We fix an Iwasawa decomposition $G=KAN$, and $\Hr^{n+1}$ identifies with the quotient manifold $G/K$; thus, $K$ is the stabilizer of some fixed point $o \in \Hr^{n+1}$. The boundary at infinity of $\Hr^{n+1}$ is denoted by $\partial \Hr^{n+1}$.

Recall the classical Busemann function, 
\[ b_\xi(x,y) = \lim_{t \to \infty} d(x,\xi_t) - d(y,\xi_t) \]
for any $x,y \in \Hr^{n+1}$, $\xi \in \partial \Hr^{n+1}$, where $t \mapsto \xi_t$ is some geodesic with positive endpoint $\xi$ (\emph{i.e.} $\xi = \lim_{t \to \infty} \xi_t$), and $d$ is the hyperbolic distance in $\Hr^{n+1}$.

Good references for this section are \cite{Roblin},\cite{LNQuint} and \cite{SchapiraPaulinPollicott}.
\subsection{Limit set and growth exponent}

Let $\Gamma$ be a discrete subgroup of $G$. If $x$ is some point of $\Hr^{n+1}$, the set of accumulation points of the orbit $\Gamma \cdot x$ on $\Hr^{n+1} \cup \partial \Hr^{n+1}$ is a subset $\Lambda_\Gamma$ of the boundary, namely $\Lambda_\Gamma = \overline{\Gamma \cdot x} \cap \partial  \Hr^{n+1}$. This set does not depend on $x$. It is called the limit set of $\Gamma$. If $\Lambda_\Gamma$ is a finite set, $\Gamma$ is called elementary, otherwise $\Gamma$ is called non-elementary.

The growth exponent of $\Gamma$,
\[ \delta_\Gamma = \limsup_{R \to \infty} \frac{1}{R} \mathrm{Card} \{\gamma \in \Gamma\ ;\ d(x,\gamma x) \leq R \} \]
does not depend on $x$. It is a finite number, $0 < \delta_\Gamma \leq n$. 

Our main interest here is in studying the geometry of Patterson-Sullivan measures. We now introduce these measures.
\subsection{Patterson-Sullivan measures}
\begin{definition}
Let $\Gamma$ be a non-elementary discrete subgroup of $G$. Let $\beta$ be some real number $\geq 0$. A $\Gamma$-conformal density of exponent $\beta$ is a family $(\mu_x)_{x \in \Hr^{n+1}}$ of finite measures on $\partial \Hr^{n+1}$ which satisfies 
\begin{enumerate}
\item $\Gamma$-equivariance: 
\[ \gamma \mu_x = \mu_{\gamma x} \]
for any $x \in \Hr^{n+1}$ and any $\gamma \in \Gamma$.
\item Conformity: for any $x,y \in \Hr^{n+1}$, $\mu_x$ and $\mu_y$ are equivalent measures and the Radon-Nikodym derivative is given by
\[ \frac{\mathrm{d} \mu_y}{\mathrm{d} \mu_x}(\xi) = e^{-\beta b_\xi(y,x)} \]
almost everywhere.
\end{enumerate}
\end{definition}

The following well-known Theorem is basic.
\begin{theorem}[\cite{Roblin}]
Let $\Gamma$ be a non-elementary discrete subgroup of $G$, with growth exponent $\delta_\Gamma$. There exist a $\Gamma$-conformal density of exponent $\delta_\Gamma$. 
\end{theorem}
If $(\mu_x)$ is such a density, the measures $\mu_x$ are called Patterson-Sullivan measures.

\subsection{The Bowen-Margulis-Sullivan measure}
We will identify the unit tangent bundle $T^1 \Hr^{n+1}$  with the quotient space $G/M$, where $M$ is the centralizer, in $K$, of the Cartan subgroup $A$ (recall that we fixed an Iwasawa decomposition $G=KAN$). For more details on this identification see \cite{Winter}. Through this isomorphism $T^1 \Hr^{n+1} \simeq G/M$, the geodesic flow on $T^1 \Hr^{n+1}$ identifies with the right operation of $A$ on $G/M$, $(gM, a) \mapsto gaM$. More precisely, there is a (unique) isomorphism $\R \to A$, $t \mapsto a_t$, such that the geodesic flow on $T^1 \Hr^{n+1}$ identifies with the operation of $\R$ on $G/M$ given by $(t,gM) \mapsto ga_tM$.

The Hopf isomorphism is the bijective mapping from $T^1 \Hr^{n+1}$ onto $\partial^2 \Hr^{n+1} \times \R$, that maps the unit tangent vector $u$ with base point $x$ to the triple
\[ (u^-,u^+, b_{u^-}(x,o)) \]
where $u^-$ and $u^+$ are the negative and positive, respectively, endpoints of the geodesic whose derivative at $t=0$ is $u$. Here we denote by $\partial^2 \Hr^{n+1}$ the set of all $(\xi,\eta) \in \partial \Hr^{n+1} \times \partial \Hr^{n+1}$ such that $\xi \neq \eta$. 

We will  write, by an abuse of language, $u=(\xi,\eta,s)$, meaning $\xi=u^-$, $\eta=u^+$ and $s=b_{u^-}(x,o)$ (where $x$ is the base point of $u$).

Now, let $\Gamma$ be a discrete non-elementary subgroup of $G$. Let $\mu$ be a $\Gamma$-conformal density of exponent $\delta_\Gamma$. Fix some arbitrary point $x \in \Hr^{n+1}$. We define the BMS measure $\BMS$ on the unit tangent bundle of $\Hr^{n+1}$, $T^1 \Hr^{n+1}$:
\begin{equation} \label{eq.BMS}
\mathrm{d} \BMS(u) = e^{\delta_\Gamma(b_\xi(x,u)+b_\eta(x,u))} \mathrm{d} \mu_x(\xi) \mathrm{d} \mu_x(\eta) \mathrm{d}s\text.
\end{equation}
This Radon measure does not depend on the choice of $x$. It is invariant under the geodesic flow as well as under $\Gamma$. Consequently, the measure on $\Gamma \backslash T^1 \Hr^{n+1}$ defined by passing to the quotient is a Radon measure that is invariant under the geodesic flow. Equivalently, we obtain a Radon measure on $\Gamma \backslash G/M$, that is $A$-invariant on the right.
\begin{definition}
We say that a discrete non-elementary subgroup $\Gamma$ of $G$ has finite BMS measure if the associated Bowen-Margulis-Measure on $\Gamma \backslash T^1 \Hr^{n+1}$  is finite.
\end{definition}

\begin{theorem}[\cite{Roblin}]
Let $\Gamma$ be a discrete non-elementary subgroup of $G$. If $\Gamma$ has finite BMS measure, the $\Gamma$-conformal density of exponent $\delta_\Gamma$ is unique, atomless, its support is the limit set $\Lambda_\Gamma$, and the conical limit set $\Lambda_\Gamma^c$ has full measure. Furthermore, the BMS measure is (strongly) mixing with respect to the geodesic flow.
\end{theorem}
We will always assume that $\Gamma$ has finite BMS measure, so that the Patterson-Sullivan measure is essentially unique and we may say, by an abuse of language,  ``let $\mu$ be \emph{the} Patterson-Sullivan measure associated to $\Gamma$''.

The Bowen-Margulis-Measure on $\Gamma \backslash G/M$ does not exactly suit our needs, because, as we said, $N$ does not act on the right on this space. Hence we are led to consider the unique $M$-invariant lifting of this measure to $\Gamma \backslash G$. We still call this measure on $\Gamma \backslash G$ the BMS measure. Note that the right action of $A$ on $\Gamma \backslash G/M$ extends to a right action on $\Gamma \backslash G$. The space $\Gamma \backslash G$ is sometimes called \emph{the frame bundle of $\Gamma \backslash \Hr^{n+1}$}, and the operation of $\R$ on $\Gamma \backslash G$, $(t, \Gamma g) \mapsto \Gamma g a_t$ is called the \emph{frame flow}.  The following Theorem is crucial.
\begin{theorem}[\cite{Winter}] \label{th.dale}
Let $\Gamma$ be a discrete non-elementary subgroup of $G$. Assume that $\Gamma$ is Zariski-dense and has finite BMS measure. Then the BMS measure on $\Gamma \backslash G$ is (strongly) mixing under the right operation of $A$.
\end{theorem}

If $g$ is any element of $G$, we let $g^+ = (gM)^+$ and $g^- = (gM)^-$ (see \emph{supra}).

\subsection{Disintegrating the Bowen-Margulis-Sullivan measure along the horospherical group}
\subsubsection{General facts}
A key point in our approach is that we look at the conditional measures of the BMS measure on $\Gamma \backslash G$ along the horospherical group $N$ (recall that we fixed an Iwasawa decomposition $G=KAN$). This is part of the general theory of conditional measures along group operations. 

In this subsection we distract from our main objective in order to recall the basic facts we will need. 

Let $X$ be  a standard Borel space  where some fixed second countable locally compact topological group $R$  acts (on the left) in a Borel way with discrete stabilizers (\emph{i.e.} for any $x \in X$ the stabilizer of $x$ is a discrete subgroup of $R$). (The reader should think of $R$ as $N$ or any $m$-plane of $N$.) 

Let $\lambda$ be some Borel probability measure on $X$.  We can disintegrate $\lambda$ along the operation of $R$ on $X$. For any $x \in X$, the ``conditional measure of $\lambda$ along $R$ at $x$'' is a projective Radon measure $\sigma(x)$.

Recall that if $\mu$ is some non-zero Radon measure on $R$, its projective class is
\[ [\mu] = \{ t \mu\ ;\ t>0 \}\]
and the space of all such projective classes is denoted by $\Mrad(R)$. The space of non-zero Radon measures on $R$ is a standard Borel space in a canonical way and we endow $\Mrad(R)$ with the quotient Borel structure, thus turning it into a standard Borel space.

Let us now describe briefly the construction of this mapping $\sigma$. For more details, we refer the reader to \cite{preprint} or \cite{thesis}.

According to a Theorem of Kechris \cite{Kechris1}, we may find a Borel subset $\Sigma$ of $X$ such that 
\begin{enumerate}
\item Any $R$-orbit in $X$ meets $\Sigma$, \emph{i.e.} $X=R\Sigma$.
\item There is a neighbourhood $U$ of the identity in $R$ such that, for any $x' \in \Sigma$, the only $g \in U$ that maps $x'$ into $\Sigma$ is the identity element.
\end{enumerate}

We call such a set a \emph{complete lacunary section}. The Borel mapping
\[ a_\Sigma : R \times \Sigma \to X \qquad (g,x') \mapsto gx' \]
has countable fibers, that is, $a_\Sigma^{-1}(x)$ is countable for any $x \in X$. We may thus define a $\sigma$-finite measure $a_\Sigma^* \lambda$ on $R \times \Sigma$ by letting
\[ \int f(g,x') \ \mathrm{d}(a_\Sigma^* \lambda)(g,x')=\int \sum_{(g,x') \in a_\Sigma^{-1}(x)} f(g,x') \ \mathrm{d}\lambda(x) \]
for any positive Borel mapping $f : R \times \Sigma \to \R$.

Choose some Borel finite measure $\lambda_\Sigma$ on $\Sigma$ such that a Borel subset $A$ of $\Sigma$ is negligible (with respect to $\lambda_\Sigma$) if and only if $R \times \Sigma$ is negligible (with respect to $a_\Sigma^*\lambda$).

We may now disintegrate $a_\Sigma^* \lambda$ above $\lambda_\Sigma$:
\[ a_\Sigma^* \lambda = \int \mathrm{d}\lambda_\Sigma(x')\ \sigma_\Sigma(x') \otimes \delta_{x'} \]
where $\delta_{x'}$ is the Dirac probability measure concentrated on $x'$ and $\sigma_\Sigma(x')$ is a Radon measure on $R$, as we may check.

\begin{proposition}[\cite{BenoistQuint} or \cite{thesis}, Proposition 2.1.1.14] \label{prop.benoistquint}
There is a mapping $\sigma : X \to \Mrad(R)$ with the following property: for any complete lacunary section $\Sigma$, there is a conegligible subset $X' \subset X$ such that for almost every $x' \in \Sigma$ and any $g \in R$, if $gx'$ belongs to $X'$ then
\[ (R_g)^{-1}\sigma(gx')=[\sigma_\Sigma(x')] \text. \]
The mapping $\sigma$ is unique up to a negligible set. Furthermore, $\sigma$ is essentially $R$-equivariant, \emph{i.e.} there is a conegligible set $X' \subset X$ such that if $x$ and $gx$ belong to $X'$, then
\begin{equation} \sigma(gx)=R_g \sigma(x) \text. \label{eq.equivariance}
\end{equation}

If $\lambda$ is finite, then $\sigma$ maps $X$ into $\Mrad(R)$, and is Borel.
\end{proposition}

By an abuse of language, we will speak of the ``measure'' $\sigma(x)$, despite the fact that this is not a genuine measure.

It is easy to check that for $\lambda$-almost every $x$, the identity element $e$ of $R$ belongs to the support of $\sigma(x)$.

Also, a subset $A \subset X$ is negligible with respect to $\lambda$ if and only if for $\lambda$-almost every $x \in X$, the set of all $g \in R$ such that $gx \in A$ is negligible with respect to $\sigma(x)$. 

We are now going to state an important transitivity property. Let $L$ be some closed  subgroup of $R$. Obviously, $L$ acts on $X$ in a Borel way with discrete stabilizers, so that we may disintegrate $\lambda$ along $L$, this gives another Borel mapping
\[ \sigma_{\lambda,L} : X \to \Mrad(L) \text. \]
Now, for $\lambda$-almost every $x$, we disintegrate also the (projective) Radon measure $\sigma(x)$ along the natural operation of $L$ on $R$ (on the left). Note that $\sigma(x)$ is (usually) not a finite (projective) measure, but there is no difficulty in generalizing the above proposition to the setting of a measure that is $\sigma$-finite instead of being finite, see \cite{thesis}. For $\lambda$-almost every $x$, let 
\[ \sigma_{\sigma(x),L} : R \to \Mrad(L) \]
be the Borel mapping obtained by disintegrating $\sigma(x)$ along $L$.

\begin{proposition}[\cite{thesis}, proposition 2.1.3.5] \label{prop.transitivity}
For $\lambda$-almost every $x$ and $\sigma(x)$-almost every $g \in R$, 
\[ \sigma_{\lambda,L}(gx)=\sigma_{\sigma(x),L}(g) \text. \]
\end{proposition}

Finally, the following lemma, although obvious, is important.
\begin{lemma} \label{lemma.isom}
Let $Y$ be another standard Borel space where $R$ acts on the left in a Borel way with discrete stabilizers. Let $\alpha$ be some Borel group automorphism of $R$. Let $\phi : X \to Y$ be some Borel automorphism that is a $\alpha$-homomorphism of $R$-spaces, which means that
\[ \phi(rx)=\alpha(r) (\phi(x)) \]
for any $x \in X$ and $r \in R$.  Denote by $\nu$ the push-forward $\phi \lambda$ of $\lambda$ through $\phi$. We disintegrate $\lambda$ and $\nu$ along $R$, this yields two mappings 
\[ \sigma_\lambda : X \to \Mrad(R),\qquad \sigma_\nu : Y \to \Mrad(R)\text{.}\]
Then for $\lambda$-almost every $x \in X$, there holds
\[ \sigma_\nu(\phi(x))=\alpha (\sigma_\lambda(x)) \text. \]
\end{lemma}

\subsubsection{Dimension and transverse dimension along a given subgroup}

Now we specialize to the setting where $X=\Gamma \backslash G$, $R=N$ (recall that we fixed an Iwasawa decomposition $G=KAN$) and $\lambda$ is the BMS measure on $\Gamma \backslash G$, $\BMS$. Fix some $m$-plane $U$ in $N$,  $1 \leq m \leq n-1$, and let $\sigma_N$ and $\sigma_U$  be the Borel mappings from $\Gamma \backslash G$ to $\Mrad(N)$ and $\Mrad(U)$, obtained by disintegrating $\BMS$ along $N$ and $U$, respectively.

\begin{lemma} \label{lemma.biglemma}
\begin{enumerate}
\item \emph{Transitivity of disintegration:} For $\BMS$-almost every $x$, if we disintegrate (in the usual sense) $\sigma_N(x)$ along the quotient mapping $\pi : N \to N/U$, yielding
\[ \sigma_N(x) = \int_{N/U} \sigma_{N,\pi g}(x)\ \mathrm{d}(\theta(x))(\pi g) \]
then for $\sigma_N(x)$-almost every $g \in N$, the conditional measure $\sigma_{N,\pi g}(x)$ of $\sigma_N(x)$ above $\pi(g)$ is the push-forward of $\sigma_U(xg)$  through the translation $L_g : U \to \pi^{-1}(g)$.

\item For $\BMS$-almost every $x = \Gamma g$ in $\Gamma \backslash G$ the push-forward of $\sigma_N(x)$ through the homeomorphism
\[ \begin{array}{rcl} N &\to& \partial \Hr^{n+1} \setminus \{g^-\}\\ h & \mapsto &(gh)^+
\end{array} \]
is equivalent to the Patterson-Sullivan measure (restricted to the complement of $g^-$), and the Radon-Nikodym derivative is a continuous mapping
\[ \partial \Hr^{n+1} \setminus \{g^- \} \to \ ]0,\infty[ \text. \]
\item For $\BMS$-almost every $x$, $\sigma_N(x)$ and $\sigma_U(x)$ are exact dimensional; their dimensions are almost surely independent of $x$. We denote the (almost sure) dimension of $\sigma_N(x)$ and $\sigma_U(x)$ by $\Dim(\BMS,N)$ and $\Dim(\BMS,U)$, respectively.
\item For any integer $m$, $1 \leq m \leq n-1$, and any $m$-plane $U$ in $N$, there is a well-defined ``transversal (lower) dimension''
\[  \DimT(\BMS,N/U) \in [0,n-m] \]
such that, for $\BMS$-almost every $x \in \Gamma \backslash G$, if we restrict $\sigma_N(x)$ to some compact neighbourhood $B$ of the identity element of $N$, and look at the projection $\pi (\sigma_N(x)|B)$, then, for $\sigma_N(x)$-almost every $g \in B$, 
\[ \diminf(\pi (\sigma_N(x)|B), \pi(g)) = \DimT(\BMS,N/U) \text. \]
\item If $\Dim(\BMS,U)$ is zero, then $\sigma_N(x)$ is almost surely concentrated on a Borel graph above $N/U$, \emph{i.e.} for $\BMS$-almost every $x \in \Gamma \backslash G$ there is a Borel section for $N/U$ that has full $\sigma_N(x)$-measure; in other words, there is a Borel subset $E \subset N$ of full measure (with respect to $\sigma_N(x)$) such that $\pi$ restricts to a bijection $\pi|E : E \to N/U$. \label{biglemma.graph}
\item The following Ledrappier-Young formula holds:
\[ \delta_\Gamma = \Dim(\BMS,U) +\DimT(\BMS,N/U) \text. \]
\end{enumerate}

\end{lemma}
\begin{proof}
Statement 1. is  just a restatement of proposition \ref{prop.transitivity}. Statement 2 is straightforward. For statements 3 to 6 see \cite{preprint} or \cite{thesis}.
\end{proof}
\begin{remark}
This lemma relies on the crucial fact that $\BMS$ is ergodic with respect to any non trivial $a \in A$ (it is, indeed, strongly mixing with respect to the frame flow, see theorem \ref{th.dale}).
\end{remark}

\subsubsection{Recurrence and dissipativity}\label{ss.recurrence}
Let $X$ be a standard Borel space endowed with a Borel $\sigma$-finite measure $\lambda$. Let $H$ be a locally compact second countable topological group acting in a Borel way on $X$, with discrete stabilizers. 

We do \emph{not} assume that $\lambda$ is invariant under $H$.

\begin{definition}
A Borel subset $W \subset X$ is a \emph{wandering set} if for $\lambda$-almost every $x \in X$, the set
\[ \{ g \in H\ ;\ gx \in W\} \]
is relatively compact.

A Borel subset $A \subset X$ is a \emph{recurrent set} if for any Borel subset $B \subset A$ and $\lambda$-almost every $x \in B$, the set
\[ \{ g \in H\ ;\ gx \in B \} \]
is not relatively compact.

We say that $\lambda$ is totally recurrent with respect to $H$ if $X$ is a recurrent set. We say that $\lambda$ is totally dissipative with respect to $H$ if $X$ is a countable union of wandering sets.
\end{definition}

It follows immediately that $A$ is a recurrent set if and only if $A \cap W$ is negliglible for every wandering set $W$.

It is possible to generalize the classical Hopf decomposition to this setting (where the measure is not even quasi-invariant), see \cite{thesis} section 2.1.2.

\begin{proposition}[\cite{thesis}, Th\'{e}or\`{e}me 2.1.2.6]
Assume that $\lambda$ is finite. Let $\sigma : X \to \Mrad(H)$ be the mapping obtained by disintegrating $\lambda$ along $H$. 
\begin{itemize}
\item If, for $\lambda$-almost every $x$, $\sigma(x)$ is \emph{finite}, then $\lambda$ is totally dissipative (with respect to $H$).
\item If, for $\lambda$-almost every $x$, $\sigma(x)$ is \emph{infinite}, then $\lambda$ is totally recurrent (with respect to $H$).
\end{itemize}
\end{proposition}

The following lemma relates the dynamics of the BMS measure (with respect to some $m$-plane $U$ in $N$) to its dimension (along $U$).

\begin{lemma}[\cite{thesis}, corollary 2.2.1.9] \label{lemma.dissip}
  Let $U$ be an $m$-plane of $N$, $1 \leq m \leq n-1$. The following assertions are equivalent:
  \begin{enumerate}
  \item $\Dim(\BMS,U)=0$;
  \item $\BMS$ is totally dissipative with respect to $U$;
  \item For $\BMS$-almost every $x \in \Gamma \backslash G$, $\sigma(x)$ is concentrated on a Borel graph over $N/U$;
    \item There is a Borel subset $\Omega \subset \Gamma \backslash G$ of full BMS measure such that, for any $x \in \Omega$, $\Omega \cap xU = \{x\}$.
    \end{enumerate}
 Also, $\Dim(\BMS,U)>0$ if and only if $\BMS$ is totally recurrent with respect to $U$.
\end{lemma}

\section{Dimension and dynamics of the Bowen-Margulis-Sullivan along some horospherical subgroup}
\subsection{Dimension and transverse dimension}
\begin{theorem} \label{th.principal1}
Let $\Gamma$ be a Zariski-dense discrete subgroup of $G$ with finite BMS measure. Fix an integer $m$, $1 \leq m \leq n-1$. For any $m$-plane $U$ in $N$, the following dichotomy holds:
\begin{enumerate}
\item If $\delta_\Gamma \leq n-m$, then $\Dim(\BMS,U) = 0$ and $\DimT(\BMS,U)=\delta_\Gamma$.
\item If $\delta_\Gamma > n-m$, then $\Dim(\BMS,U)=\delta_\Gamma-(n-m)$ and $\DimT(\BMS,U)=n-m$.
\end{enumerate}

Besides, the transversal measure is exact dimensional: for $\BMS$-almost every $x \in \Gamma \backslash G$, and any compact neighbourhood $B$ of the identity in $N$, the projection of $\sigma(x)|B$ onto $N/U$ is exact dimensional (of dimension $\DimT(\BMS,N/U)$).

Furthermore, when $\delta_\Gamma > n-m$ the transversal measure is in fact absolutely continuous: for $\BMS$-almost every $x \in \Gamma \backslash G$ and any compact neighbourhood $B$ of the identity in $N$, the projection of $\sigma(x)|B$ onto $N/U$ is absolutely continuous (with respect to the Haar measure).
\end{theorem}

\begin{proof}
Fix $m$. The first step is to prove that  $\Dim(\BMS,U)$ does not depend on the choice of the $m$-plane $U$. Indeed choose two $m$-plane $U,U'$ in $N$. Pick some $g \in M$ such that $U'=gUg^{-1}$ (recall that $M$ acting on $N$ by conjugation is just $\mathbf{SO}(n)$ acting on $\R^n$). For $\BMS$-almost every $x \in \Gamma \backslash G$, 
\[ \sigma(xg)=\mathrm{Int}(g) \left(\sigma(x)\right) \]
where $\mathrm{Int}(g)(y)=gyg^{-1}$ for $y \in N$ (this is a basic consequence of $M$-invariance of $\BMS$ and uniqueness of conditional measures).
The measure automorphism $R_g : \Gamma \backslash G \to \Gamma \backslash G$ mapping $x$ to $xg^{-1}$ intertwines the (right) operation of $U$ with the (right) operation of $U'$: for any $u \in U$, if we let $u'=gug^{-1}$,
\[ R_g(x) u' = xg^{-1}(gug^{-1}) = xug^{-1} = R_g(xu) \text. \]
Consequently, using the fact that $R_g (\BMS)=\BMS$, we get
\[ \sigma_U(x) = \mathrm{Int}(g) (\sigma_{U'}(xg^{-1})) \]
and since $\mathrm{Int}(g)$ is an isometry $U \to U'$, this implies readily that
\[ \Dim(\BMS,U)=\Dim(\BMS,U') \text. \]

Likewise, one shows in the same way that $\DimT(\BMS,N/U)$ does not depend on the choice of the $m$-plane $U$. Also, assume that the $m$-plane $U$ satisfies the following property: there is  some relatively compact neighbourhood of the identity $B$ in $N$ such that the push-forward
\[ \frac{\sigma(x)|B}{\sigma(x)(B)} \]
through the quotient mapping $N \to N/U$ is absolutely continuous for $\BMS$-almost every $x \in \Gamma \backslash G$. Then the same property holds for any other $m$-plane $U'$.

Now let $B$ be some fixed relatively compact neighbourhood of the identity in $N$. For any $x \in \Gamma \backslash G$, let
\[ \nu^B(x) = \frac{\sigma(x)|B}{\sigma(x)(B)}\]
this is a finite measure on $N$. For $\BMS$-almost every $x$, this measure is exact dimensional, of dimension $\delta_\Gamma$. Fix such an $x$. For almost every $m$-plane $U$ in $N$, the push-forward of $\nu^B(x)$ through the quotient mapping $N \to N/U$ must be exact dimensional of dimension $\inf \{ n-m,\delta_\Gamma\}$. 

On the other hand, for any $m$-plane $U$ and $\BMS$-almost every $x$, the push-forward of $\nu^B(x)$ through $N \to N/U$ has lower dimension equal to $\DimT(\BMS,N/U)$ almost everywhere.

By virtue of Fubini's Theorem, we deduce that 
\begin{enumerate}
\item For any $m$-plane $U$, 
\[ \DimT(\BMS,U) = \inf\{n-m,\delta_\Gamma\} \]
\item For any $m$-plane $U$ and $\BMS$-almost every $x \in \Gamma \backslash G$, the push-forward of $\nu^B(x)$ through $N \to N/U$ is exact dimensional (of dimension $\DimT(\BMS,U)$.
\item If $\delta_\Gamma > n-m$, then for any $m$-plane $U$ and $\BMS$-almost every $x \in \Gamma \backslash G$, the push-forward of $\nu^B(x)$ through $N \to N/U$ is absolutely continuous.
\end{enumerate} 
Finally, the Ledrappier-Young Theorem immediately implies that for any $m$-plane $U$, 
\[ \Dim(\BMS,U)=\left\{ \begin{array}{lcl} 0 & \mathrm{if} & \delta_\Gamma \leq n-m \\ \delta_\Gamma-(n-m) & \mathrm{otherwise} &
  \end{array}\right. \]
\end{proof}

In proving this Theorem we used Marstrand's Theorem (proposition \ref{prop.marstrand}) and the fact that the BMS measure on $\Gamma \backslash G$ is $M$-invariant (on the right). 
The reader should pay attention to the order of quantifiers in Theorem \ref{th.principal1}. This Theorem really deals with projections of the BMS measure in almost every direction. Our argument revolves around Marstrand's Theorem and uses Winter's Theorem about ergodicity of the BMS measure on $\Gamma \backslash G$ (with respect to frame flow), as well as Ledrappier-Young's Theorem.

We will soon (section \ref{sect.dimproj}) state and prove a result about the projection in some fixed direction -- for this we will need another tool. For now, the following corollary is of interest.
\begin{corollary}
Assume that $\delta_\Gamma \leq n-m$. For almost every $\xi \in \partial \Hr^{n+1}$, the following holds: if we identify the $\partial \Hr^{n+1} \setminus \{\xi\}$ with the $n$-space $N=\R^n$ (through stereographic projection), then, for almost every $m$-plane $U$ in $N$, the Patterson-Sullivan measure is concentrated on a Borel graph along $U$ (\emph{i.e.} there is a Borel section $N/U \to N$ whose range has full Patterson-Sullivan measure).
\end{corollary}
Thus we may say that if $\delta_\Gamma \leq n-m$, the Patterson-Sullivan is concentrated on a Borel graph along almost every $m$-plane.
\begin{proof}
See lemma \ref{lemma.biglemma}.\ref{biglemma.graph}.
\end{proof}

\subsection{Dynamics of  the Bowen-Margulis-Sullivan  measure}
We still assume that $\Gamma$ is a Zariski-dense discrete subgroup of $G$ with finite BMS measure. The following corollary follows from Theorem \ref{th.principal1} by virtue of lemma \ref{lemma.dissip}.
\begin{corollary}
Fix some $m$-plane $U$ in $N$ and let $U$ act on $\Gamma \backslash G$ (on the right). With respect to the BMS measure  this operation is
\begin{enumerate}
\item totally dissipative if $\delta_\Gamma \leq n-m$
\item totally recurrent if $\delta_\Gamma > n-m$.
\end{enumerate}
\end{corollary}

\section{Dimension of projections} \label{sect.dimproj}
\subsection{Statement}
\begin{definition}
Let $\mu$ be some Radon measure on $\R^n$. We say that $\mu$ is regular if
\begin{enumerate}
\item it is exact dimensional of dimension $\delta \in [0,n]$ 
\item  for any $k$-plane $V$ in $\R^n$ ($1 \leq k \leq n-1$), and any relatively compact open subset $B \subset \R^n$ such that $\mu(B)>0$, the orthogonal projection of $\mu|B$ onto $V$ has dimension
\[ \inf\{ \delta,k\} \text. \]
\end{enumerate}
\end{definition}

\begin{theorem} \label{th.principal2}
Let $\Gamma$ be a discrete Zariski-dense subgroup of $G$ with finite BMS measure. Let $\sigma : \Gamma \backslash G \to \Mrad(N)$ be the mapping obtained by disintegrating $\BMS$ along $N$. For $\BMS$-almost every $x \in \Gamma \backslash G$, $\sigma(x)$ is a regular measure.
\end{theorem}

\subsection{Proof}
Let $\sigma : \Gamma \backslash G \to \Mrad(N)$ be the mapping obtained by disintegrating $\BMS$ along $N$. We are going to apply Theorem \ref{th.hochman}. Identify $N$ with $\R^n$ and denote by $B_1$ the unit cube $[-1,1]^n$. The space $\mathcal M$ of non-zero Radon measures on $\R^n$ identifies with the space of non-zero Radon measures on $N$.

For any $x \in \Gamma \backslash G$ let
\[ \sigma^*(x) = \frac{\sigma(x)}{\sigma(x)(B_1)} \]
that is, $\sigma^*$ is the composition of $\sigma$ and the mapping $[\mu] \mapsto \mu^*$. Likewise, we let
\[ \sigma^\square(x)=\frac{\sigma(x)|B_1}{\sigma(x)(B_1)} \text. \]
Obviously $\sigma^*(x) \in \mathcal M^*$ and $\sigma^\square(x) \in \mathcal M^\square$.

We define a distribution $P$ on $\mathcal M^*$:
\[ P = \int \mathrm{d} \BMS(x)\ \delta_{\sigma^*(x)}  \]
where $\delta_{\sigma^*(x)}$ is the Dirac mass at $\sigma^*(x)$.
The proof of our Theorem consists of the following checks:
\begin{enumerate}
\item The distribution $P$ is an ergodic fractal distribution.
\item For $\BMS$-almost every $x \in \Gamma \backslash G$ $\sigma^\square(x)$ is a uniformly scaling measure generating the restricted version of $P$ (that is, $P^\square$).
\item For any $k$, $1 \leq k \leq n-1$, and any linear mapping $\pi$ from $\R^n$ onto $\R^k$, 
\[ E_{P^\square}(\pi) = \inf\{ \delta_\Gamma,k\} \text. \]
\end{enumerate}

These three points are enough to prove the Theorem. Indeed, by virtue of Theorem \ref{th.hochman}, for $\BMS$-almost every $x \in \Gamma \backslash G$, and \emph{every} $k$-plane $V \subset N$, we obtain
\[ \diminf\left( \pi_V(\sigma^\square(x)) \right) \geq \inf\{ \delta_\Gamma,k\} \]
where $\pi_V$ is the orthogonal projection $\R^n \to V$. The converse inequality is obviously true, so that we get 
\[ \diminf\left( \pi_V(\sigma^\square(x)) \right) = \inf\{ \delta_\Gamma,k\} \]
for \emph{every} $k$-plane $V$.

Before proceeding to prove the above three points, let us clarify notations. Remember that we fixed an isomorphism $\R \to A$, $t \mapsto a_t$ such that the frame flow on $\Gamma \backslash G$ identifies with the operation $(t,\Gamma g) \mapsto \Gamma g a_t$. The  automorphism $\mathrm{Int}(a_t)$ of $N$ that maps $g$ to $a_t g a_{-t}$ identifies with the homothety
\[ S_t : \R^n \to \R^n,\quad y \mapsto e^t y \text. \]

For any fixed $t \in \R$ and $\BMS$-almost every $x \in \Gamma \backslash G$, $\sigma^*(xa_{-t}) = S_t \sigma^*(x)$ by lemma \ref{lemma.isom}.

Also, for $\BMS$-almost every $x \in \Gamma \backslash G$ and $\sigma(x)$-almost every $g \in N$, the push-forward of $\sigma(x)$ through the mapping $L(g^{-1}) : h \mapsto g^{-1}h$, is equal to $\sigma(xg)$.

\begin{lemma}
The distribution $P$ is an ergodic fractal distribution.
\end{lemma}
\begin{proof}
The mapping
\[ \begin{array}{ccc}\Gamma \backslash G & \to & \mathcal{M}^* \\ x & \mapsto & \sigma^*(x)
\end{array} \]
 maps $\BMS$ onto $P$ and   intertwines ($\BMS$-almost surely) the operation of $(a_t)_t$ on $\Gamma \backslash G$ with the operation of $(S_t^*)_t$ on $\mathcal M^*$. In other words, for every $t \in \R$ and $\BMS$-almost every $x \in \Gamma \backslash G$, there holds
\[ S_t^* \sigma^*(x) = \sigma^*(x a_t)\text. \]
The dynamical system $(\mathcal{M}^*, P, (S_t^*)_t)$ is thus a factor of the dynamical system $(\Gamma \backslash G, \BMS, (a_t)_t)$. Since the latter is ergodic, the former must be ergodic as well.

Let us now check quasi-Palmness. Let $\mathcal E$ be some Borel subset of $\mathcal M^*$ and let $E$ be the set of all $x \in \Gamma \backslash G$ such that $\sigma^*(x) \in \mathcal E^*$. Let $W$ be some relatively compact neighbourhood of the identity in $N$. Keeping in mind the relation \eqref{eq.equivariance} in proposition \ref{prop.benoistquint}, one readily checks that  $\mathcal E$ is negligible with respect to the distribution on $\mathcal M^*$ 
\[ \int \mathrm{d} P(\mu) \int_W \mathrm{d} \mu(v) \delta_{T_v^* \mu} \]
if and only if the set of all $x \in \Gamma \backslash G$ such that
\[ \sigma^*(x)\{ y \in W; xy \in E \} > 0 \]
is negligible with respect to $\BMS$. This is equivalent to the fact that $E$ is negligible with respect to $\BMS$, hence also to the fact that $\mathcal E$ is negligible with respect to $P$.
\end{proof}

\begin{lemma} For $\BMS$-almost every $x \in \Gamma \backslash G$, the measure $\sigma^*(x)$ is uniformly scaling and generates $P^\square$.
\end{lemma}

\begin{proof}
Let $h$ be some continuous mapping on the compact second countable space $\mathcal M^\square$. We must prove that for $\BMS$-almost every $x \in \Gamma \backslash G$, for $\sigma^\square(x)$-almost every $g \in N$, 
\[ (*) \qquad \lim_{T \to \infty} \frac{1}{T} \int_0^T h\left( S_t^\square (L_{g^{-1}} \sigma(x)) \right) \mathrm{d} t = \int h(\sigma^\square(x)) \mathrm{d} \BMS (x) \mathrm{.} \]
Let us first remark that this equality holds in the particular case when $g$ is the identity element of $N$. Indeed, since $h$ is bounded, we need just check that
\[ \lim_{k \to \infty} \frac{1}{k} \sum_{i=0}^{k-1} \int_0^1 h(S_t^\square \sigma^* (x a_{-i})) \mathrm{d} t \]
for $\BMS$-almost every $x$. This is a consequence of pointwise ergodic Theorem of Birhov applied to the function
\[ x \mapsto \int_0^1 h(S_t^\square \sigma^*(x)) \mathrm{d} t \]
(recall that $\BMS$ is ergodic with respect to the automorphism $x \mapsto xa_{-1}$, see Theorem \ref{th.dale}).
Now let us denote by $X$ the set of all points $x \in \Gamma \backslash G$ such that $(*)$ is satisfied when $g$ is the identity element of $N$. For $\BMS$-almost every $x \in \Gamma \backslash G$ and $\sigma(x)$-almost every $g \in N$, there holds $L(g^{-1}) \sigma(x)=\sigma(xg)$; furthermore, $xg$ belongs to $X$ for $\sigma(x)$-almost every $g$, since $X$ has full measure with respect to $\BMS$. 

This shows that $(*)$ holds for $\BMS$-almost every $x \in \Gamma \backslash G$ and $\sigma^\square(x)$-almost every $g \in N$.

Hence the lemma.
\end{proof}

\begin{lemma}
Let $V$ be some $k$-plane in $N$, $1 \leq k \leq n-1$, and $\pi$ be the orthogonal projection from $N$ onto $V$. Then
\[ E_{P^\square}(\pi) = \inf \{ \delta_\Gamma,k\} \text. \]
\end{lemma}
\begin{proof}
By the very definition of $P$, it stands to reason that this distribution is $M$-invariant (since $\BMS$ is $M$-invariant). Now apply corollary \ref{corollary.lol}.
\end{proof}

\bibliography{bibli}

\end{document}